\documentclass{amsart}

\usepackage{amsmath}
\usepackage{amsthm}
\usepackage{amsfonts}
\usepackage{amssymb}

\usepackage{graphics}
\usepackage{epsfig}
\usepackage{color}
\usepackage{verbatim}

\renewcommand\Re{{\operatorname{Re}}}

\newcommand{\vertiii}[1]{{\left\vert\kern-0.25ex\left\vert\kern-0.25ex\left\vert #1 
    \right\vert\kern-0.25ex\right\vert\kern-0.25ex\right\vert}}

\theoremstyle{plain}
  \newtheorem{theorem}{Theorem}

  \newtheorem{lemma}[theorem]{Lemma}
  \newtheorem{corollary}[theorem]{Corollary}

\theoremstyle{definition}

  \newtheorem{remark}[theorem]{Remark}
\date{1st June 2014}
\title{Flow monotonicity and Strichartz inequalities}
\begin{document}
\author{Jonathan Bennett}
\address{School of Mathematics, University of Birmingham, Edgbaston, Birmingham, B15 2TT, United Kingdom}
\email{j.bennett@bham.ac.uk, m.iliopoulou@bham.ac.uk}
\thanks{This work was supported by the European Research Council [grant
number 307617] (Bennett, Iliopoulou)}

\author{Neal Bez}
\address{Department of Mathematics, Graduate School of Science and Engineering,
Saitama University, Saitama 338-8570, Japan}
\email{nealbez@mail.saitama-u.ac.jp}

\author{Marina Iliopoulou}
%\address{Department of Mathematics, University of Birmingham, Edgbaston Birmingham B15 2TT}
%\email{m.iliopoulou@bham.ac.uk}

\begin{abstract} We identify complete monotonicity properties underlying a variety of well-known sharp Strichartz inequalities in euclidean space.
\end{abstract}
\maketitle
\section{Introduction}
A wide variety of important inequalities in analysis and geometry may be understood as a consequence of the monotonicity of associated functionals along appropriate flows. Perhaps the simplest and best-known example is the Cauchy--Schwarz inequality on $L^2(\mathbb{R}^d)$, which is easily seen to follow from the observation that the quantity
$$
t\mapsto \int_{\mathbb{R}^d} (e^{t\Delta}|f_1|^2)^{1/2}(e^{t\Delta}|f_2|^2)^{1/2}
$$
is nondecreasing, and thanks to the quadratic nature of the flows, converges to $\|f_1\|_2\|f_2\|_2$ as $t\rightarrow\infty$.
Such monotonicity phenomena are quite revealing, often allowing the identification of extremisers and sharp constants for the associated inequalities. We refer the reader to the survey articles \cite{BEsc} and \cite{CM} for further discussion of this diverse and far-reaching theory.

In \cite{BBCH} it was observed that such a monotonicity phenomenon exists in the context of the Strichartz estimates
\begin{equation}\label{strichschrod}
\|e^{is\Delta}f\|_{L^p_sL^q_x(\mathbb{R}\times\mathbb{R}^d)}\leq C_{p,q}\|f\|_{L^2(\mathbb{R}^d)}; \;\;\;\; \frac{2}{p}+\frac{d}{q}=\frac{d}{2};\;\; 2\leq p,q\leq\infty;\;\; (p,q,d)\not= (2,\infty,2)
\end{equation}
for the linear time-dependent Schr\"odinger equation. In particular, it was shown that the quantity
\begin{equation}\label{one}
t\mapsto\|e^{is\Delta}(e^{t\Delta}|f|^2)^{1/2}\|_{L^p_sL^q_x(\mathbb{R}\times\mathbb{R}^d)}
\end{equation}
is nondecreasing whenever
$q\in 2\mathbb{N}$ and $q$ divides $p$; i.e. for $(p,q,d)=(6,6,1), (8,4,1), (4,4,2)$.
This elementary result, while quite striking, has a number of shortcomings. In particular, it does not appear to extend easily to other dispersive/wave equations, other flows, the broader scale of Sobolev--Strichartz estimates or more general exponents $p$, $q$. The purpose of this paper is to show that much of this rigidity may be overcome if the underlying flows are chosen to be linear rather than quadratic.
%\footnote{We do not, however, overcome the more fundamental obstacle that the Lebesgue exponents $p\in 2\mathbb{N}$.}.
For example, we shall see in Section 2 that the quantity
\begin{equation}\label{two}
t\mapsto C_{p,q}^p\|e^{t\Delta}f\|_{L^2(\mathbb{R}^d)}^p - \|e^{is\Delta}e^{t\Delta}f\|_{L^p_sL^q_x(\mathbb{R}\times\mathbb{R}^d)}^p
\end{equation}
is nonincreasing whenever $q\in 2\mathbb{N}$ and $q$ divides $p$. Here $C_{p,q}$ denotes the constant in \eqref{strichschrod}. Moreover, provided one is prepared to replace the constant $C_{p,q}$ with a larger one, the arithmetic restriction that $q\in 2\mathbb{N}$ and $q$ divides $p$ may be completely removed.
As may be expected, the monotonicity of both quantities \eqref{one} and \eqref{two} may be obtained by differentiating with respect to $t$ followed by a careful application of the divergence theorem in spatial variables. However, by contrast with \eqref{one}, the linear flow in \eqref{two} permits an alternative Fourier-analytic approach via Plancherel's theorem. This feature allows us to identify flows and similar monotone quantities associated with a variety of well-known sharp Strichartz and Sobolev--Strichartz estimates in the setting of other dispersive and wave equations. For example, in the setting of the wave equation we see that the heat extension $e^{t\Delta}f$ should be replaced with the harmonic extension $e^{-tD}f$ where $D=\sqrt{-\Delta}$. Furthermore, these monotone quantities turn out to be \emph{completely monotone} (or \emph{totally monotone}) in the sense that their $k$th derivatives have sign $(-1)^k$ for all $k\in\mathbb{N}$.

It seems plausible that the perspectives of this paper may bear fruit in more general contexts where explicit expressions for solutions are not available. Notice that, for example, the monotonicity of \eqref{two} may be rewritten as the monotonicity of the expression
\begin{equation}\label{two'}
t\mapsto C_{p,q}^p\left(\int_{\mathbb{R}^d}|u(0,t,x)|^2 \,\mathrm{d}x\right)^{p/2} - \int_{\mathbb{R}}\left(\int_{\mathbb{R}^d}|u(s,t,x)|^q \,\mathrm{d}x\right)^{p/q} \mathrm{d}s
\end{equation}
where $u:\mathbb{R}\times (0,\infty)\times\mathbb{R}^d\rightarrow\mathbb{C}$ satisfies the equations
\begin{equation} \label{e:spacetimeSchro}
\frac{\partial u}{\partial t}=i\frac{\partial u}{\partial s}=\Delta_x u.
\end{equation}
As we shall see in Section \ref{sect2'}, this permits arguments using mainly calculus, provided one is prepared to sacrifice the sharp constant $C_{p,q}$ in \eqref{strichschrod}.

\emph{Organisation.} Section \ref{sect2} is devoted to such monotonicity results in the context of sharp Strichartz inequalities for dispersive and wave equations in settings where Fourier-analytic methods are available. In Section \ref{sect2'} we make some observations about monotonicity without the arithmetic restriction $q\in 2\mathbb{N}$ and $q|p$. In Section \ref{sect3} we identify similar elementary monotonicity phenomena associated with the closely-related Stein--Tomas restriction theorem for general curved surfaces in low dimensions. Finally, in Section \ref{sect4} we make some observations about such monotonicity in the setting of the Strichartz inequalities for the kinetic transport equation, along with certain intimately related $k$-plane transform inequalities.

It is interesting to contrast the results and perspectives of this paper with those of Planchon and Vega \cite{PV}; see also the section on monotonicity formulae in \cite{Tao}.

\section{Monotonicity and sharp Strichartz inequalities for dispersive equations}\label{sect2}
In this section we reveal the complete monotonicity of several functionals associated with sharp space-time estimates, including certain classical Strichartz estimates, for the Schr\"odinger, wave and Klein--Gordon propagators. In the Schr\"odinger case we have the most to say, and so it is here that we will begin.

Throughout this section, we use the notation $\mathfrak{I}_m(f)$ for the functional given by
\[
\mathfrak{I}_m(f) = \int_{(\mathbb{R}^d)^{m}} |\widehat{\Pi(f)}(\xi)|^2 K(\xi) \, \mathrm{d}\xi
\]
for appropriate $f : \mathbb{R}^d \to \mathbb{C}$ and $m \geq 2$. Here, $K : (\mathbb{R}^d)^m \to [0,\infty)$ is some nonnegative kernel, and $\Pi$ is some transformation which takes functions on $\mathbb{R}^d$ to functions on $(\mathbb{R}^d)^m$. For each of the main results in this section, the pair $(\Pi, K)$ will be different, and for brevity we suppress the dependence on $(\Pi,K)$ in the notation for $\mathfrak{I}_m$.

The class of measurable functions $f : \mathbb{R}^d \to \mathbb{C}$ for which $\mathfrak{I}_m(f)$ is finite will be denoted by $\Upsilon_m$. We will use the Fourier transform heavily in this section, and since we are tracking explicit constants, we clarify that the Fourier transform we use is
\[
\widehat{f}(\xi) = \int_{\mathbb{R}^d} f(x) e^{-i x \cdot \xi} \, \mathrm{d}x
\]
for appropriate $f : \mathbb{R}^d \to \mathbb{C}$. We also use the notation $\widetilde{F}$ for the space-time Fourier transform of appropriate $F : \mathbb{R} \times \mathbb{R}^d \to \mathbb{C}$.

\subsection{The Schr\"odinger equation}
\begin{theorem} \label{t:QSchro}
Let $d \geq 2$, $m \geq 2$ and $f \in \Upsilon_m$. Then the function $Q : (0,\infty) \to \mathbb{R}$ given by
\[
Q(t) = \emph{S}(m,d) \mathfrak{I}_m(e^{t\Delta}f) - \| e^{is\Delta} e^{t\Delta}f \|_{L^{2m}(\mathbb{R} \times \mathbb{R}^{d})}^{2m}
\]
is completely monotone (decreasing). Here, the constant $\emph{S}(m,d)$ is given by
\[
\emph{S}(m,d) = \frac{|\mathbb{S}^{(m-1)d-1}|}{2m^{\frac{dm-2}{2}}(2\pi)^{(2m-1)d-1}},
\]
$\Pi(f)$ is the $m$-fold tensor product
\begin{equation*}
\Pi(f) = f \otimes \cdots \otimes f
\end{equation*}
and
\begin{equation*}
K(\xi) = \bigg( \sum_{1 \leq i < j \leq m} |\xi_i - \xi_j|^2 \bigg)^{\tfrac{1}{2}(d(m-1) - 2)}.
\end{equation*}
\end{theorem}
The exponent on the kernel $\tfrac{1}{2}(d(m-1) - 2)$ is nonnegative whenever $d \geq 2$ and $m \geq 2$, and vanishes if and only if $(m,d) \in \{(3,1),(2,2)\}$. Before going further, we immediately state the following corollary, which concerns the classical Strichartz estimates which are contained in Theorem \ref{t:QSchro}, or quickly follow from it.
\begin{corollary} \label{c:StrichartzSchro}
Suppose $(p,q,d) \in \{(6,6,1),(8,4,1),(4,4,2)\}$ and $f \in L^2(\mathbb{R}^d)$. Then the function $Q : (0,\infty) \to \mathbb{R}$ given by
\[
Q(t) = C_{p,q}^p \|e^{t\Delta}f\|_{L^2(\mathbb{R}^d)}^p - \| e^{is\Delta} e^{t\Delta}f \|_{L^{p}_sL^{q}_x(\mathbb{R} \times \mathbb{R}^{d})}^{p}
\]
is completely monotone (decreasing), where
\[
C_{p,q} = \left\{\begin{array}{lllll} 12^{-1/12}\,, & (p,q,d) = (6,6,1) \\
2^{-1/4}\,, &  (p,q,d) = (8,4,1) \\
2^{-1/2}\,, & (p,q,d) = (4,4,2)
\end{array} \right..
\]
\end{corollary}
The underlying inequality for the monotone quantity $Q$ in Theorem \ref{t:QSchro}, obtained by comparing $Q(t)$ in the limiting cases $t \to 0$ and $t \to \infty$, states that
\begin{equation} \label{e:Carneiro}
\| e^{is\Delta} f \|_{L^{2m}(\mathbb{R} \times \mathbb{R}^d)}^{2m} \leq \mbox{S}(m,d) \int_{(\mathbb{R}^d)^{m}} \prod_{k=1}^m |\widehat{f}(\xi_k)|^2 \bigg( \sum_{1 \leq i < j \leq m} |\xi_i - \xi_j|^2 \bigg)^{\alpha}  \, \mathrm{d}\xi,
\end{equation}
where $\alpha = \tfrac{1}{2}(d(m-1) - 2)$. The constant $\mbox{S}(m,d)$ is optimal and certain gaussian functions form the set of extremisers. In full generality, inequality \eqref{e:Carneiro} and the characterisation of extremisers is due to Carneiro \cite{Carneiro}, and in the cases $(m,d) \in \{(3,1),(2,2)\}$ corresponding to $\alpha = 0$, these facts were established earlier by Foschi \cite{Foschi} and Hundertmark--Zharnitsky \cite{HZ}. Focussing on the case $m=2$, \eqref{e:Carneiro} simplifies to
\begin{equation} \label{e:Carneirobilinear}
\| e^{is\Delta} f \|_{L^{4}(\mathbb{R} \times \mathbb{R}^d)}^{4} \leq \frac{|\mathbb{S}^{d-1}|}{2^d(2\pi)^{3d-1}} \int_{(\mathbb{R}^d)^{2}}  |\widehat{f}(\xi_1)|^2 |\widehat{f}(\xi_2)|^2 |\xi_1 - \xi_2|^{d-2}   \, \mathrm{d}\xi
\end{equation}
which may be viewed as a sharp relative of the classical Sobolev--Strichartz estimate
\begin{equation} \label{e:SobStr}
\| e^{is\Delta} f \|_{L^{4}(\mathbb{R} \times \mathbb{R}^d)} \lesssim \|f\|_{\dot{H}^{\frac{d-2}{4}}(\mathbb{R}^d)}\,.
\end{equation}
Ozawa--Tsutsumi \cite{OT} established the following sharp inequality which may also be interpreted as relative of \eqref{e:SobStr}, requiring only $L^2$ regularity on the initial data, and compensating for this by measuring the solution in a classical homogeneous Sobolev space of nonpositive order
\begin{equation} \label{e:OT}
\| (-\Delta)^{\frac{2-d}{4}} |e^{is\Delta}f|^2 \|_{L^2(\mathbb{R} \times \mathbb{R}^d)}^2 \leq \frac{|\mathbb{S}^{d-1}|}{4(2\pi)^{d-1}} \|f\|_{L^2(\mathbb{R}^d)}^4 \,.
\end{equation}
Here $d \geq 2$, the constant is optimal, and again this has gaussian extremisers (it was observed in \cite{OT} that gaussians are amongst the extremisers; see \cite{BBJP} for the full characterisation). Estimates \eqref{e:Carneirobilinear} and \eqref{e:OT} bring to mind important work of Beals \cite{Beals} and Klainerman--Machedon \cite{KM1}, \cite{KM2} on null form estimates in the context of the wave equation.

Our next result shows that the sharp Ozawa--Tsutsumi inequality \eqref{e:OT} also enjoys complete monotonicity under heat-flow.
\begin{theorem} \label{t:QOT}
Let $d \geq 2$ and $f \in L^2(\mathbb{R}^d)$. Then the function $Q : (0,\infty) \to \mathbb{R}$ given by
\[
Q(t) = \frac{|\mathbb{S}^{d-1}|}{4(2\pi)^{d-1}} \|e^{t\Delta}f\|_{L^2(\mathbb{R}^d)}^4 - \| (-\Delta)^{\frac{2-d}{4}} |e^{is\Delta}e^{t\Delta}f|^2 \|_{L^{2}(\mathbb{R} \times \mathbb{R}^{d})}^{2}
\]
is completely monotone (decreasing).
\end{theorem}
We shall see that a nice feature of our proofs of Theorems \ref{t:QSchro} and  \ref{t:QOT} is that they follow the same fundamental steps; this can be viewed as a unification of the underlying sharp inequalities, whose previously known proofs were rather different.

\begin{proof}[Proof of Theorem \ref{t:QSchro}]
Using Fourier inversion and by multiplying out the $L^{2m}$ norm we obtain
\begin{align*}
\| e^{is\Delta} e^{t\Delta} f \|_{L^{2m}(\mathbb{R} \times \mathbb{R}^{d})}^{2m}  & =  \int_{\mathbb{R}^{d+1}} \bigg|  \frac{1}{(2\pi)^d} \int_{\mathbb{R}^d} e^{ix \cdot \xi} e^{-is|\xi|^2} e^{-t|\xi|^2} \widehat{f}(\xi) \, \mathrm{d}\xi \bigg|^{2m} \, \mathrm{d}x\mathrm{d}s \\
&  = \frac{1}{(2\pi)^{(2m-1)d-1}} \int_{(\mathbb{R}^d)^m} \int_{(\mathbb{R}^d)^m} e^{-2t|\xi|^2} \widehat{\Pi(f)}(\xi) \overline{\widehat{\Pi(f)}(\eta)} \, \mathrm{d}\Sigma_\xi(\eta) \mathrm{d}\xi,
\end{align*}
where, for each $\xi \in (\mathbb{R}^d)^m$, $\mathrm{d}\Sigma_\xi$ is the measure
$$
\mathrm{d}\Sigma_\xi(\eta) = \delta\bigg(\sum_{j=1}^m \xi_j - \sum_{j=1}^m \eta_j \bigg)\delta\bigg(\sum_{j=1}^m |\xi_j|^2 - \sum_{j=1}^m |\eta_j|^2\bigg) \mathrm{d}\eta.
$$
\begin{lemma} \label{l:P1Schro}
For all $\xi \in (\mathbb{R}^d)^m$ we have
\begin{equation} \label{e:P1Schro}
\int_{(\mathbb{R}^d)^m} \mathrm{d}\Sigma_\xi  = \frac{|\mathbb{S}^{(m-1)d-1}| }{2m^{\frac{dm-2}{2}}}  K(\xi).
\end{equation}
\end{lemma}
\begin{proof}
Let the measure $\mathrm{d}\mu$ be given by $\mathrm{d}\mu(\sigma,\zeta) = \delta(\sigma - |\zeta|^2)$. Then, the mass of $\mathrm{d}\Sigma_\xi$ is $\mathrm{d}\mu^{(m)}(\sigma,\zeta)$, where
\[
(\sigma,\zeta) = \bigg(\sum_{j=1}^m |\xi_j|^2, \sum_{j=1}^m \xi_j\bigg)
\]
and $\mathrm{d}\mu^{(m)}$ is the $m$-fold convolution of the measure $\mathrm{d}\mu$ with itself. Since $\mathrm{d}\mu$ is invariant under the affine map
$$
(\sigma,\zeta) \mapsto (\sigma + 2\zeta \cdot v + |v|^2, \zeta + v)
$$
for fixed $v \in \mathbb{R}^d$, it follows that
$$
\mathrm{d}\mu^{(m)}(\sigma,\zeta) = \mathrm{d}\mu^{(m)}(\sigma + 2\zeta \cdot v + m|v|^2, \zeta + mv).
$$
Choosing $v = -\frac{\zeta}{m}$ we obtain
$$
\mathrm{d}\mu^{(m)}(\sigma,\zeta) = \mathrm{d}\mu^{(m)}(\sigma - \tfrac{1}{m}|\zeta|^2,0) = (\sigma - \tfrac{1}{m}|\zeta|^2)^{\frac{1}{2}(d(m-1)-2)}\mathrm{d}\mu^{(m)}(1,0),
$$
where the second equality follows by a simple change of variables. The constant $\mathrm{d}\mu^{(m)}(1,0)$ can be computed using polar coordinates; finally we obtain
\[
\mathrm{d}\mu^{(m)}(\sigma,\zeta) = \frac{|\mathbb{S}^{(m-1)d-1}|}{2m^{\frac{d}{2}}} (\sigma - \tfrac{1}{m}|\zeta|^2)^{\frac{1}{2}(d(m-1)-2)}
\]
and using the identity
\[
m \sum_{j=1}^m |\xi_j|^2 - \bigg| \sum_{j=1}^m \xi_j \bigg|^2 = \sum_{1 \leq i < j \leq m} |\xi_i - \xi_j|^2
\]
we obtain the claimed expression for the mass of $\mathrm{d}\Sigma_\xi$.
\end{proof}
\begin{remark}
The above proof of Lemma \ref{l:P1Schro} generalises the argument of Foschi in Lemmas 3.2 and 4.1 of \cite{Foschi} which covered the cases $(m,d) \in \{(3,1),(2,2)\}$.
\end{remark}
 Using Lemma \ref{l:P1Schro}, we may write
\begin{align*}
\mbox{S}(m,d) \mathfrak{I}_m(e^{t\Delta}f) & = \mbox{S}(m,d) \int_{(\mathbb{R}^d)^m} e^{-2t|\xi|^2} |\widehat{\Pi(f)}(\xi)|^2 K(\xi) \, \mathrm{d}\xi \\
& = \frac{1}{(2\pi)^{(2m-1)d-1}}  \int_{(\mathbb{R}^d)^m} \int_{(\mathbb{R}^d)^m} e^{-2t|\xi|^2} |\widehat{\Pi(f)}(\xi)|^2  \, \mathrm{d}\Sigma_\xi(\eta)\mathrm{d}\xi
\end{align*}
and since $\mathrm{d}\Sigma_\xi(\eta)\mathrm{d}\xi = \mathrm{d}\Sigma_\eta(\xi)\mathrm{d}\eta$ and $|\xi|^2 = |\eta|^2$ on the support of this measure, it follows that
\[
\mbox{S}(m,d) \mathfrak{I}_m(e^{t\Delta}f) = \frac{1}{2(2\pi)^{(2m-1)d-1}}  \int_{(\mathbb{R}^d)^m} \int_{(\mathbb{R}^d)^m} e^{-2t|\xi|^2} (|\widehat{\Pi(f)}(\xi)|^2 + |\widehat{\Pi(f)}(\eta)|^2)  \, \mathrm{d}\Sigma_\xi(\eta)\mathrm{d}\xi.
\]
Therefore
\[
Q(t) = \frac{1}{2(2\pi)^{(2m-1)d-1}}  \int_{(\mathbb{R}^d)^m} \int_{(\mathbb{R}^d)^m} e^{-2t|\xi|^2} |\widehat{\Pi(f)}(\xi) - \widehat{\Pi(f)}(\eta)|^2 \, \mathrm{d}\Sigma_\xi(\eta)\mathrm{d}\xi
\]
and $Q$ is manifestly completely monotone (decreasing).
\end{proof}
\begin{proof}[Proof of Corollary \ref{c:StrichartzSchro}]
It is helpful to write
\[
Q_{p,q,d}[f](t) = C_{p,q}^p \|e^{t\Delta}f\|_{L^2(\mathbb{R}^d)}^p - \| e^{is\Delta} e^{t\Delta}f \|_{L^{p}_sL^{q}_x(\mathbb{R} \times \mathbb{R}^{d})}^{p}
\]
for the quantity $Q(t)$ in the statement of Corollary \ref{c:StrichartzSchro}.

When $(m,d) \in \{(3,1),(2,2)\}$, Plancherel's theorem implies
$$
\mathfrak{I}_m(f) = (2\pi)^{dm} \|f\|_{L^2(\mathbb{R}^d)}^{2m}
$$
and the complete monotonicity of $Q_{6,6,1}[f](t)$ and $Q_{4,4,2}[f](t)$ immediately follows from Theorem \ref{t:QSchro}. Finally, using the identity
\[
\|e^{is\Delta} e^{t\Delta} f \|_{L^8_sL^4_x(\mathbb{R} \times \mathbb{R})}^8 = \| e^{is\Delta}e^{t\Delta}(f \otimes f)\|_{L^4(\mathbb{R} \times \mathbb{R}^2)}^4,
\]
where $\Delta$ is either one-dimensional or two-dimensional depending on the context, we obtain
\[
Q_{8,4,1}[f](t) = Q_{4,4,2}[f \otimes f](t)
\]
and the complete monotonicity of $Q_{8,4,1}[f](t)$ follows.
\end{proof}

\begin{proof}[Proof of Theorem \ref{t:QOT}]
Using Plancherel's theorem in the space-time variables we obtain
\begin{equation*}
\| (-\Delta)^{\frac{2-d}{4}} |e^{is\Delta}e^{t\Delta}f|^2 \|_{L^2(\mathbb{R} \times \mathbb{R}^d)}^2 = \frac{1}{(2\pi)^{d+1}} \int_{\mathbb{R}^{d+1}} |\zeta|^{2-d} |\widetilde{|e^{is\Delta}e^{t\Delta}f|^2}(\tau,\zeta)|^2 \, \mathrm{d}\zeta \mathrm{d}\tau.
\end{equation*}
We have
$$
\widetilde{e^{is\Delta}e^{t\Delta}f}(\tau,\zeta) = 2\pi \delta(\tau + |\zeta|^2) \widehat{e^{t\Delta}f}(\zeta) = 2\pi \delta(\tau + |\zeta|^2) e^{-t|\zeta|^2}\widehat{f}(\zeta)
$$
and therefore
\begin{align*}
\widetilde{|e^{is\Delta}e^{t\Delta}f|^2}(\tau,\zeta) & = \frac{1}{(2\pi)^{d+1}} \,\widetilde{e^{is\Delta}e^{t\Delta}f} * \widetilde{\overline{e^{is\Delta}e^{t\Delta}f}} \,(\tau,\zeta) \\
& = \frac{1}{(2\pi)^{d-1}} \int_{(\mathbb{R}^{d})^2} e^{-t|\xi|^2}  \widehat{f}(\xi_1) \widehat{\overline{f}}(\xi_2) \delta(\tau+|\xi_1|^2-|\xi_2|^2) \delta(\zeta - \xi_1 -\xi_2) \,\mathrm{d}\xi.
\end{align*}
Expanding the $L^2$ norm gives
\begin{align*}
&  \| (-\Delta)^{\frac{2-d}{4}} |e^{is\Delta}e^{t\Delta}f|^2 \|_{L^2(\mathbb{R} \times \mathbb{R}^d)}^2 \\
& =  \frac{1}{(2\pi)^{3d-1}} \int_{(\mathbb{R}^{d})^2} \int_{(\mathbb{R}^{d})^2} e^{-t(|\xi|^2 + |\eta|^2)} |\xi_1 + \xi_2|^{2-d} \widehat{f}(\xi_1) \widehat{\overline{f}}(\xi_2)  \overline{\widehat{f}(\eta_1) \widehat{\overline{f}}(\eta_2) }  \,\, \times \\
& \qquad  \qquad\qquad  \qquad\qquad  \qquad \delta(-|\xi_1|^2 + |\xi_2|^2 + |\eta_1|^2-|\eta_2|^2) \delta(\xi_1 + \xi_2 - \eta_1 -\eta_2) \,\mathrm{d}\xi\mathrm{d}\eta
\end{align*}
and a relabelling of the variables $(\xi_1,\eta_1,\xi_2,\eta_2) \to (\xi_1,\eta_1,\eta_2,\xi_2)$ implies
\begin{align*}
\| (-\Delta)^{\frac{2-d}{4}} |e^{is\Delta}e^{t\Delta}f|^2 \|_{L^2(\mathbb{R} \times \mathbb{R}^d)}^2 = \frac{1}{(2\pi)^{3d-1}} \int_{(\mathbb{R}^{d})^2} \int_{(\mathbb{R}^{d})^2}  e^{-2t|\xi|^2} \widehat{\Pi(f)}(\xi) \overline{\widehat{\Pi(f)}(\eta)} \, \mathrm{d}\Sigma_\xi(\eta) \mathrm{d}\xi
\end{align*}
where
\[
\Pi(f) = f \otimes f(- \,\cdot)
\]
and
\[
\mathrm{d}\Sigma_\xi(\eta)  =  |\xi_1 + \eta_2|^{2-d} \delta(|\eta_1|^2 + |\eta_2|^2-|\xi_1|^2-|\xi_2|^2) \delta(\eta_1  -\eta_2 - (\xi_1 - \xi_2))  \mathrm{d}\eta.
\]
Note that $\Pi(f)$ and $\mathrm{d}\Sigma_\xi$ are not the same as in Theorem \ref{t:QSchro}; importantly, however, we do have the following analogue of Lemma \ref{l:P1Schro}.
\begin{lemma} \label{l:P1OT}
For each $\xi \in (\mathbb{R}^{d})^2$ we have
\begin{equation*}
\int_{(\mathbb{R}^{d})^2} \mathrm{d}\Sigma_\xi = \frac{|\mathbb{S}^{d-1}|}{4}\,.
\end{equation*}
\end{lemma}
\begin{proof}
The change of variables $(\zeta_1,\zeta_2) = (\eta_1+\xi_2,\eta_2+\xi_1)$ gives
\begin{align*}
\int_{(\mathbb{R}^{d})^2} \mathrm{d}\Sigma_\xi(\eta)  & = \int_{(\mathbb{R}^{d})^2}  |\zeta_2|^{2-d} \delta(|\zeta_1 - \xi_2|^2 + |\zeta_2 - \xi_1|^2-|\xi_1|^2-|\xi_2|^2) \delta(\zeta_1  -\zeta_2)  \,\mathrm{d}\zeta \\
& = \frac{1}{2} \int_{\mathbb{R}^{d}}  |\zeta_2|^{2-d} \delta(|\zeta_2|^2 - \zeta_2 \cdot (\xi_1 + \xi_2)) \,\mathrm{d}\zeta_2
\end{align*}
and changing to polar coordinates gives
\begin{align*}
\int_{(\mathbb{R}^{d})^2} \mathrm{d}\Sigma_\xi(\eta)  = \frac{1}{2} \int_{\mathbb{S}^{d-1}} \int_0^\infty  \delta(r - \omega \cdot (\xi_1 + \xi_2)) \mathrm{d}r \mathrm{d}\omega = \frac{|\mathbb{S}^{d-1}| }{4} 
\end{align*}
as claimed.
\end{proof}
By Lemma \ref{l:P1OT} and the symmetry relation $\mathrm{d}\Sigma_\eta(\xi)\mathrm{d}\eta = \mathrm{d}\Sigma_\xi(\eta)\mathrm{d}\xi$,
\begin{align*}
\frac{(2\pi)^{2d}|\mathbb{S}^{d-1}|}{4}  \|e^{t\Delta}f\|_{L^2(\mathbb{R}^d)}^4  & =  \int_{(\mathbb{R}^{d})^2}\int_{(\mathbb{R}^{d})^2} e^{-2t|\xi|^2} |\widehat{\Pi(f)}(\xi)|^2   \, \mathrm{d}\Sigma_\xi(\eta) \mathrm{d}\xi \\
& = \frac{1}{2} \int_{(\mathbb{R}^{d})^2}\int_{(\mathbb{R}^{d})^2} e^{-2t|\xi|^2} \big( |\widehat{\Pi(f)}(\xi)|^2 + |\widehat{\Pi(f)}(\eta)|^2 \big)   \, \mathrm{d}\Sigma_\xi(\eta) \mathrm{d}\xi
\end{align*}
from which we obtain
\begin{align*}
Q(t) = \frac{1}{2(2\pi)^{3d-1}}  \int_{(\mathbb{R}^{d})^2}\int_{(\mathbb{R}^{d})^2}  e^{-2t|\xi|^2}  |\widehat{\Pi(f)}(\xi) - \widehat{\Pi(f)}(\eta)|^2 \, \mathrm{d}\Sigma_\xi(\eta) \mathrm{d}\xi
\end{align*}
and hence $Q$ is completely monotone (decreasing).
\end{proof}

In the forthcoming subsections, we prove analogous results to Theorem \ref{t:QSchro} (and Corollary \ref{c:StrichartzSchro}) for the wave and Klein--Gordon equations.

\subsection{The wave equation}
We consider pairs $(m,d)$ such that $d \geq 3$ and $m \geq 2$, or $d=2$ and $m \geq 3$. Define an associated exponent $\beta \geq 0$ by
\[
\beta = \beta(m) = \frac{(d-1)(m-1) - 2}{2}
\]
and a constant $\mbox{A}(m,d)$ by
\begin{equation*}
\mbox{A}(2,d) = \frac{|\mathbb{S}^{d-1}|}{2^{d-2}}
\end{equation*}
and for $m \geq 3$,
\begin{equation*}
\mbox{A}(m,d) = \frac{|\mathbb{S}^{d-1}|^{m-1}}{2^{2\beta(m) + 1}} \prod_{j=2}^{m-1} \mbox{B}(d-1,\beta(j) + 1) .
\end{equation*}
Here, $\mbox{B}(x,y) = \int_0^1 r^{x-1}(1-r)^{y-1}\,\mathrm{d}r$ is the beta function.
\begin{theorem} \label{t:Qwave}
Suppose either $d \geq 3$ and $m \geq 2$, or $d=2$ and $m \geq 3$. For any $f \in \Upsilon_m$ the function $Q : (0,\infty) \to \mathbb{R}$ given by
\[
Q(t) = \emph{W}(m,d) \mathfrak{I}_m(e^{-tD}f) - \| e^{isD} e^{-tD} f \|_{L^{2m}(\mathbb{R} \times \mathbb{R}^{d})}^{2m}
\]
is completely monotone (decreasing). Here, the constant $\emph{W}(m,d)$ is given by
\[
\emph{W}(m,d) = \frac{2^{\beta(m)}}{(2\pi)^{(2m-1)d-1}}\emph{A}(m,d),
\]
$\Pi(f)$ is the $m$-fold tensor product
\begin{equation*}
\Pi(f) = D^{\frac{1}{2}} f \otimes \cdots \otimes D^{\frac{1}{2}}f
\end{equation*}
and
\begin{equation*}
K(\xi) = \bigg( \sum_{1 \leq i<j \leq m} (|\xi_i||\xi_j| - \xi_i \cdot
\xi_j) \bigg)^{\frac{1}{2}((d-1)(m-1) - 2)}
\end{equation*}
\end{theorem}
Recall that we are using the familiar notation $D = \sqrt{-\Delta}$. In the special cases $(m,d) \in \{(3,2),(2,3)\}$ where $\beta = 0$ we have
$$
\mathfrak{I}_m(f) = (2\pi)^{dm}\|f\|_{\dot{H}^{\frac{1}{2}}(\mathbb{R}^d)}^{2m},
$$
where $\dot{H}^{\frac{1}{2}}(\mathbb{R}^d)$ denotes the homogeneous Sobolev space of order $\frac{1}{2}$, and Theorem \ref{t:Qwave} immediately yields the following.
\begin{corollary} \label{c:Strichartzwave}
Suppose $(p,d) \in \{(6,2),(4,3)\}$ and $f \in \dot{H}^{\frac{1}{2}}(\mathbb{R}^d)$. Then the function $Q : (0,\infty) \to \mathbb{R}$ given by
\[
Q(t) = C_{p}^p \|e^{-tD}f\|_{\dot{H}^{\frac{1}{2}}(\mathbb{R}^d)}^p - \| e^{isD} e^{-tD}f \|_{L^{p}(\mathbb{R} \times \mathbb{R}^{d})}^{p}
\]
is completely monotone (decreasing), where
\[
C_4 = (2\pi)^{-1/4} \qquad \text{and} \qquad C_6 = (2\pi)^{-1/6}.
\]
\end{corollary}
The sharp Strichartz inequalities obtained by comparing $Q(t)$ as $t \to 0$ and $t \to \infty$ in Corollary \ref{c:Strichartzwave} are
\begin{equation} \label{e:Foschiwave2}
\| e^{isD} f \|_{L^{6}(\mathbb{R} \times \mathbb{R}^{2})} \leq \frac{1}{(2\pi)^{\frac{1}{6}}} \|f\|_{\dot{H}^{\frac{1}{2}}(\mathbb{R}^2)}
\end{equation}
and
\begin{equation} \label{e:Foschiwave3}
\| e^{isD} f \|_{L^{4}(\mathbb{R} \times \mathbb{R}^{3})} \leq \frac{1}{(2\pi)^{\frac{1}{4}}} \|f\|_{\dot{H}^{\frac{1}{2}}(\mathbb{R}^3)}
\end{equation}
and these were first obtained by Foschi \cite{Foschi}. In \cite{Foschi}, a full characterisation of the extremisers was also found; in particular, the initial data
\begin{equation*}
f(x) = (1+ |x|^2)^{-\frac{1}{2}(d-1)}
\end{equation*}
is extremal for \eqref{e:Foschiwave2} and \eqref{e:Foschiwave3}, with $d=2$ and $d=3$, respectively. Equivalently, $f$ such that
\begin{equation*}
\widehat{f}(\xi) = \frac{1}{|\xi|} e^{-|\xi|}
\end{equation*}
is extremal for both \eqref{e:Foschiwave2} and \eqref{e:Foschiwave3}, and the relevance of the flow $e^{-tD}f$ becomes more apparent.

For general $d$ and $m$ under consideration in Theorem \ref{t:Qwave}, the underlying sharp inequality is an analogue of \eqref{e:Carneiro} for the wave equation and this was proved in \cite{BR}.

\begin{proof}[Proof of Theorem \ref{t:Qwave}]
We have
\begin{align*}
\| e^{isD} e^{-tD} f \|_{L^{2m}(\mathbb{R} \times \mathbb{R}^{d})}^{2m}  & =  \int_{\mathbb{R}^{d+1}} \bigg|  \frac{1}{(2\pi)^d} \int_{\mathbb{R}^d} e^{ix \cdot \xi} e^{is|\xi|} e^{-t|\xi|} \widehat{f}(\xi) \, \mathrm{d}\xi \bigg|^{2m} \, \mathrm{d}x\mathrm{d}s \\
&  = \frac{1}{(2\pi)^{(2m-1)d-1}} \int_{(\mathbb{R}^d)^m} \int_{(\mathbb{R}^d)^m} e^{-2t\sum_{j=1}^m |\xi_j|} \widehat{\Pi(f)}(\xi) \overline{\widehat{\Pi(f)}(\eta)} \, \mathrm{d}\Sigma_\xi(\eta) \mathrm{d}\xi,
\end{align*}
where
$$
\mathrm{d}\Sigma_\xi(\eta) = \delta\bigg(\sum_{j=1}^m \xi_j - \sum_{j=1}^m \eta_j \bigg)\delta\bigg(\sum_{j=1}^m |\xi_j| - \sum_{j=1}^m |\eta_j|\bigg) \prod_{j=1}^m |\eta_j|^{-\frac{1}{2}} |\xi_j|^{-\frac{1}{2}}  \mathrm{d}\eta.
$$
\begin{lemma} \label{l:P1Wave}
If
\[
\Phi(\xi,\eta) = \left(\frac{\prod_{j=1}^m |\xi_j|}{\prod_{j=1}^m |\eta_j|} \right)^{\frac{1}{2}}
\]
then
\begin{equation} \label{e:P1wave}
\int_{(\mathbb{R}^d)^m} \Phi \,\mathrm{d}\Sigma_\xi = 2^{\beta} \emph{A}(m,d) K(\xi).
\end{equation}
\end{lemma}
\begin{proof}
Lemma 3.1 in \cite{BR} implies that
\[
\int_{(\mathbb{R}^d)^m} \prod_{j=1}^m |\eta_j|^{-1} \delta\bigg(\sigma - \sum_{j=1}^m |\eta_j|\bigg) \delta\bigg(\zeta - \sum_{j=1}^m \eta_j \bigg) \, \mathrm{d}\eta = \mbox{A}(m,d) (\sigma^2 - |\zeta|^2)^{\beta}
\]
and applying this with
\[
(\sigma,\zeta) = \bigg(\sum_{j=1}^m |\xi_j|, \sum_{j=1}^m \xi_j\bigg)
\]
we obtain \eqref{e:P1wave}.
\end{proof}
By Lemma \ref{l:P1Wave},
\begin{align*}
\mbox{W}(m,d) \mathfrak{I}_m(e^{-tD}f) & = \mbox{W}(m,d) \int_{(\mathbb{R}^d)^m} e^{-2t\sum_{j=1}^m |\xi_j|} |\widehat{\Pi(f)}(\xi)|^2 K(\xi)\, \mathrm{d}\xi \\
%& =   \frac{\mbox{W}(m,d)}{2^{2\beta}\mbox{A}(m,d)}  \int_{(\mathbb{R}^d)^m} \int_{(\mathbb{R}^d)^m} e^{-2t\sum_{j=1}^m |\xi_j|} |\widehat{F}(\xi)|^2 \Phi(\xi,%\eta) \, \mathrm{d}\Sigma_\xi(\eta) \,\mathrm{d}\xi \\
& = \frac{1}{(2\pi)^{(2m-1)d-1}}   \int_{(\mathbb{R}^d)^m} \int_{(\mathbb{R}^d)^m} e^{-2t\sum_{j=1}^m |\xi_j|} |\widehat{\Pi(f)}(\xi)|^2 \Phi(\xi,\eta) \, \mathrm{d}\Sigma_\xi(\eta) \mathrm{d}\xi
\end{align*}
and since $\Phi(\eta,\xi) = \Phi(\xi,\eta)^{-1}$, it follows that $\mbox{W}(m,d) \mathfrak{I}_m(e^{-tD}f)$ coincides with
\[
\frac{1}{2(2\pi)^{(2m-1)d-1}}  \int_{(\mathbb{R}^d)^m} \int_{(\mathbb{R}^d)^m} e^{-2t\sum_{j=1}^m|\xi_j|} (|\widehat{\Pi(f)}(\xi)|^2\Phi(\xi,\eta) + |\widehat{\Pi(f)}(\eta)|^2\Phi(\xi,\eta)^{-1})  \, \mathrm{d}\Sigma_\xi(\eta) \mathrm{d}\xi.
\]
Hence
\[
Q(t) = \frac{1}{2(2\pi)^{(2m-1)d-1}}  \int_{(\mathbb{R}^d)^m} \int_{(\mathbb{R}^d)^m} e^{-2t\sum_{j=1}^m |\xi_j|} |\widehat{\Pi(f)}(\xi)\Phi(\xi,\eta)^{\frac{1}{2}} - \widehat{\Pi(f)}(\eta)\Phi(\xi,\eta)^{-\frac{1}{2}}|^2 \, \mathrm{d}\Sigma_\xi(\eta) \mathrm{d}\xi
\]
and $Q$ is completely monotone (decreasing).
\end{proof}

\subsection{The Klein--Gordon equation} In the case $m=2$ we can also prove an analogous monotonicity phenomenon for the Klein--Gordon propagator $e^{is\sqrt{1-\Delta}}$. To state this, it is convenient to introduce the notation
\[
\phi(\varrho) = \sqrt{1 + \varrho^2}.
\]
\begin{theorem} \label{t:QKG}
Let $d \geq 2$ and $f \in \Upsilon_2$. Then the function $Q : (0,\infty) \to \mathbb{R}$ given by
\[
Q(t) = \frac{|\mathbb{S}^{d-1}|}{2^{\frac{d-1}{2}} (2\pi)^{3d-1}} \mathfrak{I}_2(e^{-t\sqrt{1-\Delta}}f) - \| e^{is\sqrt{1-\Delta}} e^{-t\sqrt{1-\Delta}} f \|_{L^{4}(\mathbb{R} \times \mathbb{R}^{d})}^{4}
\]
is completely monotone (decreasing). Here, $\Pi(f)$ is the tensor product
\begin{equation*}
\Pi(f) = \phi(D)^{\frac{1}{2}}  f \otimes \phi(D)^{\frac{1}{2}}f
\end{equation*}
and
\begin{equation*}
K(\xi) = \frac{(\phi(|\xi_1|)\phi(|\xi_2|) - \xi_1 \cdot \xi_2 - 1)^{\frac{d-2}{2}}}{(\phi(|\xi_1|)\phi(|\xi_2|) - \xi_1 \cdot \xi_2 + 1)^{\frac{1}{2}}}.
\end{equation*}
\end{theorem}
\begin{corollary} \label{c:StrichartzKG}
Suppose $d \in \{2,3\}$ and $f \in H^{\frac{1}{2}}(\mathbb{R}^d)$. Then the function $Q_0 : (0,\infty) \to \mathbb{R}$ given by
\[
Q_0(t) =  C_d^4 \| e^{-t\sqrt{1-\Delta}}f \|_{H^{\frac{1}{2}}(\mathbb{R}^d)}^4 - \| e^{is\sqrt{1-\Delta}} e^{-t\sqrt{1-\Delta}} f \|_{L^{4}(\mathbb{R} \times \mathbb{R}^{d})}^{4}
\]
is completely monotone (decreasing), where
\[
C_2 = 2^{-\frac{1}{4}} \qquad \text{and} \qquad C_3 = (2\pi)^{-\frac{1}{4}}.
\]
\end{corollary}
Here, $H^{\frac{1}{2}}(\mathbb{R}^d)$ denotes the inhomogeneous Sobolev space of order $\frac{1}{2}$. The sharp Strichartz inequalities contained in Corollary \ref{c:StrichartzKG} by comparing $Q_0(t)$ as $t \to 0$ and $t \to \infty$ were first proved by Quilodr\'an \cite{Q}. There are no extremisers for these sharp inequalities; however, if
\begin{equation} \label{e:KGextremisers}
\widehat{f}_a(\xi) = \frac{1}{\sqrt{1 + |\xi|^2}}e^{-a\sqrt{1 + |\xi|^2}}
\end{equation}
then $(f_a)$ is an extremising sequence as $a \to \infty$ and $a \to 0+$ for $d=2$ and $d=3$, respectively. Again, we see a connection with the associated flow $e^{-t\sqrt{1 - \Delta}}f$.

For general $d \geq 2$, the underlying sharp inequality in Theorem \ref{t:QKG} (associated with $Q$ rather than $Q_0$) was proved by Jeavons in \cite{Jeavons}, which does have extremisers; for example, the functions given in \eqref{e:KGextremisers}.

\begin{proof}[Proof of Theorem \ref{t:QKG}]
We have
\begin{align*}
\| e^{is\sqrt{1-\Delta}} e^{-t\sqrt{1-\Delta}} f \|_{L^{4}(\mathbb{R} \times \mathbb{R}^{d})}^{4}   = \frac{1}{(2\pi)^{3d-1}} \int_{(\mathbb{R}^d)^2} \int_{(\mathbb{R}^d)^2} e^{-2t\sum_{j=1}^2 \phi(|\xi_j|)} \widehat{\Pi(f)}(\xi) \overline{\widehat{\Pi(f)}(\eta)} \, \mathrm{d}\Sigma_\xi(\eta) \mathrm{d}\xi,
\end{align*}
where
$$
\mathrm{d}\Sigma_\xi(\eta) = \delta\bigg(\sum_{j=1}^2 \xi_j - \sum_{j=1}^2 \eta_j \bigg)\delta\bigg(\sum_{j=1}^2 \phi(|\xi_j|) - \sum_{j=1}^2 \phi(|\eta_j|)\bigg) \prod_{j=1}^2 \phi(|\eta_j|)^{-\frac{1}{2}} \phi(|\xi_j|)^{-\frac{1}{2}}  \mathrm{d}\eta.
$$
\begin{lemma} \label{l:P1KG}
If
\[
\Phi(\xi,\eta) = \left(\frac{\phi(|\xi_1|)\phi(|\xi_2|)}{\phi(|\eta_1|)\phi(|\eta_2|)} \right)^{\frac{1}{2}}
\]
then
\begin{equation} \label{e:P1KG}
\int_{(\mathbb{R}^d)^2} \Phi \,\mathrm{d}\Sigma_\xi = \frac{|\mathbb{S}^{d-1}|}{2^{\frac{d-1}{2}}} K(\xi).
\end{equation}
\end{lemma}
\begin{proof}
Lemma 1 in \cite{Jeavons} implies
\[
\int_{(\mathbb{R}^d)^2} \prod_{j=1}^2 \phi(|\eta_j|)^{-1} \delta\bigg(\sigma - \sum_{j=1}^2 \phi(|\eta_j|)\bigg) \delta\bigg(\zeta - \sum_{j=1}^2 \eta_j \bigg) \, \mathrm{d}\eta = \frac{|\mathbb{S}^{d-1}|}{2^{d-2}}  \frac{(\sigma^2 - |\zeta|^2 - 4)^{\frac{d-2}{2}}}{(\sigma^2 - |\zeta|^2)^{\frac{1}{2}}}
\]
and applying this with $(\sigma,\zeta) = (\phi(|\xi_1|) + \phi(|\xi_2|), \xi_1 + \xi_2)$ gives the desired conclusion.
\end{proof}
From Lemma \ref{l:P1KG} it follows that
\[
Q(t) = \frac{1}{2(2\pi)^{3d-1}}  \int_{(\mathbb{R}^d)^2} \int_{(\mathbb{R}^d)^2} e^{-2t\sum_{j=1}^2 \phi(|\xi_j|)} |\widehat{\Pi(f)}(\xi)\Phi(\xi,\eta)^{\frac{1}{2}} - \widehat{\Pi(f)}(\eta)\Phi(\xi,\eta)^{-\frac{1}{2}}|^2 \, \mathrm{d}\Sigma_\xi(\eta)\mathrm{d}\xi
\]
and $Q$ is completely monotone (decreasing).
\end{proof}

\begin{proof}[Proof of Corollary \ref{c:StrichartzKG}]
It suffices to show that the difference
$$
R(t) = \frac{2^{\frac{d-1}{2}} (2\pi)^{3d-1}}{|\mathbb{S}^{d-1}|}  (Q_0(t) - Q(t))
$$  
is completely monotone (decreasing). A straightforward calculation shows that
\begin{align*}
R(t) = \int_{(\mathbb{R}^d)^2} e^{-2t(\phi(|\xi_1|) + \phi(|\xi_2|))} \phi(|\xi_1|) \phi(|\xi_2|) |\widehat{f}(\xi_1)|^2  |\widehat{f}(\xi_2)|^2 (\widetilde{C}_d - K(\xi) ) \, \mathrm{d}\xi,
\end{align*}
where $\widetilde{C}_2 = \frac{1}{\sqrt{2}} $ and $\widetilde{C}_3 = 1$. It is easy to check that for each $d = 2,3$ we have the pointwise inequality $K(\xi) \leq \widetilde{C}_d$ for each $\xi \in (\mathbb{R}^d)^2$, and the completely monotonicity of $R$, and hence $Q_0$, follows.
\end{proof}

\section{Monotonicity by PDE methods}\label{sect2'}
The approach in Section \ref{sect2}, being heavily Fourier-analytic, relied crucially on the Lebesgue exponents involved being even integers. As we saw, in such situations we are able to identify monotonicity properties underlying Strichartz inequalities with sharp constants. In this section we observe that if one is prepared to sacrifice sharp constants, then monotonicity may be found for quite general Lebesgue exponents.
\begin{theorem} If $(p,q,d)$ are Schr\"odinger admissible then there exists a constant $c\geq C_{p,q}$ such that
\begin{equation*}
t\mapsto c^p
\|e^{t\Delta}f\|_{L^2(\mathbb{R}^d)}^p - \|e^{is\Delta}e^{t\Delta}f\|_{L^p_sL^q_x(\mathbb{R}\times\mathbb{R}^d)}^p
\end{equation*}
is nonincreasing.
\end{theorem}
\begin{proof}
Consider the quantity
$$
Q(t)=c^p\Bigl(\int_{\mathbb{R}^d}|u(0,t,x)|^2 \,\mathrm{d}x\Bigr)^{p/2}-\int_{\mathbb{R}}\Bigl(\int_{\mathbb{R}^d}|u(s,t,x)|^q \,\mathrm{d}x\Bigr)^{p/q} \,\mathrm{d}s,
$$
where $u : \mathbb{R} \times (0,\infty) \times \mathbb{R}^d \to \mathbb{C}$ satisfies \eqref{e:spacetimeSchro}. Differentiating with respect to $t$ we find that
$$
p^{-1}Q'(t) = c^p\|u_0\|_2^{p-2}\Re\int_{\mathbb{R}^d}
\overline{u}_0\Delta u_0 -\Re\int_{\mathbb{R}}\Bigl(\int_{\mathbb{R}^d}|u|^q \,\mathrm{d}x\Bigr)^{p/q-1}\Bigl(\int_{\mathbb{R}^d}|u|^{q-2}\overline{u}\Delta u \,\mathrm{d}x\Bigr) \,\mathrm{d}s.
$$
Here, we are suppressing the $t$-dependence on the right-hand side and $u_0$ is the initial data of the Schr\"odinger evolution (i.e. $u_0(t,x) = u(0,t,x)$). Integrating by parts we obtain
\begin{eqnarray}\label{dualityidentity}
\begin{aligned}
p^{-1}Q'(t)&=\int_{\mathbb{R}}\Bigl(\int_{\mathbb{R}^d}|u|^q \,\mathrm{d}x\Bigr)^{p/q-1}\left\{\int_{\mathbb{R}^d}|u|^{q-2}|\nabla u|^2 \,\mathrm{d}x+\left(\frac{q-2}{4}\right)\int_{\mathbb{R}^d}
|u|^{q-4}|\nabla(|u|^2)|^2 \,\mathrm{d}x\right\} \,\mathrm{d}s\\&\;\;\;\;\;\;\;\;-c^p\|u_0\|_2^{p-2}\|\nabla u_0\|_2^2\\
%&=-\int_{\mathbb{R}}\Bigl(\int_{\mathbb{R}^d}|u|^qdx\Bigr)^{p/q-1}\left\{\int_{\mathbb{R}^d}|u|^{q}\Bigl|\frac{\nabla u}{u}\Bigr|^2dx+\left(\frac{q-2}{4}\right)\int_{\mathbb{R}^d}
%|u|^{q}\Bigl|\frac{\nabla(|u|^2)}{|u|^2}\Bigr|^2dx\right\}ds\\&\;\;\;\;\;\;\;\;+c\|u_0\|_2^{p-2}\|\nabla u_0\|_2^2.
\end{aligned}
\end{eqnarray}
Since
$$
\int_{\mathbb{R}^d}
|u|^{q-4}|\nabla(|u|^2)|^2 \,\mathrm{d}x\lesssim \int_{\mathbb{R}^d}|u|^{q-2}|\nabla u|^2 \,\mathrm{d}x,
$$
it suffices to show that
$$
\int_{\mathbb{R}}\Bigl(\int_{\mathbb{R}^d}|u|^q \,\mathrm{d}x\Bigr)^{p/q-1}\Bigl(\int_{\mathbb{R}^d}|u|^{q-2}|\nabla u|^2 \,\mathrm{d}x\Bigr) \,\mathrm{d}s\lesssim\|u_0\|_2^{p-2}\|\nabla u_0\|_2^2,
$$
which in turn would follow from
\begin{equation}\label{preholder}
\int_{\mathbb{R}}\Bigl(\int_{\mathbb{R}^d}|u|^q \,\mathrm{d}x\Bigr)^{\frac{p-2}{q}}\Bigl(\int_{\mathbb{R}^d}|\nabla u|^q \,\mathrm{d}x\Bigr)^\frac{2}{q}\,\mathrm{d}s\lesssim\|u_0\|_2^{p-2}\|\nabla u_0\|_2^2
\end{equation}
by an application of H\"older's inequality in the second inner integral. However, \eqref{preholder} follows from a further application of H\"older's inequality in the variable $s$, followed by two applications of the Strichartz inequality \eqref{e:Strichartz}. Here we have used the fact that if $u$ solves the Schr\"odinger equation then so does $\nabla u$.
\end{proof}
We remark that when $p=q$ the identity \eqref{dualityidentity} identifies the Sobolev--Strichartz inequality
$$
\int_{\mathbb{R}}\left\{\int_{\mathbb{R}^d}|u|^{q}\left|\frac{\nabla u}{u}\right|^2 \,\mathrm{d}x+\left(\frac{q-2}{4}\right)\int_{\mathbb{R}^d}
|u|^{q}\left|\frac{\nabla(|u|^2)}{|u|^2}\right|^2 \,\mathrm{d}x\right\} \,\mathrm{d}s\leq c^p\|u_0\|_2^{p-2}\|\nabla u_0\|_2^2$$
as a certain ``dual" form of the classical Strichartz inequality
$$
\int_{\mathbb{R}}\int_{\mathbb{R}^d}|u|^p \,\mathrm{d}x \mathrm{d}s\leq c^p\|u_0\|_2^p.
$$
As we saw in Section \ref{sect2}, if $q\in 2\mathbb{N}$ and $q$ divides $p$ at least, both of the above inequalities may be established with the sharp constant $c=C_{p,q}$.

As we remarked in the introduction, it is conceivable that the above argument may be adapted to more general settings, such as Strichartz estimates on Riemannian manifolds (see, for example \cite{HTW}).

\section{Monotonicity and the Stein--Tomas restriction theorem in low dimensions}\label{sect3}
As we have seen, the Fourier-analytic methods in Section \ref{sect2} apply equally well in the setting of various different dispersive and wave equations. It is therefore reasonable to expect similar monotonicity statements concerning more general dispersive (pseudo-differential) equations. A convenient context for such an analysis is that of the Stein--Tomas restriction theorem. For simplicity we restrict our attention to two-dimensional surfaces in $\mathbb{R}^3$, although as will be apparent, similar statements are possible for curves in the plane.

Consider the Fourier extension operator
$$
\mathcal{E}g(x,s):=\int_Ug(\xi)e^{i(s\phi(\xi)+x\cdot\xi)} \,\mathrm{d}\xi
$$
associated with the smooth graphing function $\phi:U\rightarrow [0,\infty)$, where $U$ is some compact subset of $\mathbb{R}^2$.
We assume that $\phi$ graphs a smooth convex surface $S$ with everywhere nonvanishing curvature.

Now, for $g\in L^2(U)$ and $t\geq 0$ let $g_t(\xi)=e^{-t\phi(\xi)}g(\xi)$.

\begin{theorem} There exists a constant $c<\infty$ such that the function $Q:(0,\infty)\rightarrow \mathbb{R}$ given by
$$
Q(t)=c\|g_t\|_{L^2(U)}^4-\|\mathcal{E}g_t\|_{L^4_{x,s}(\mathbb{R}^2\times\mathbb{R})}^4
$$
is completely monotone (decreasing).
\end{theorem}

\begin{proof}
For $F:U\times U\rightarrow\mathbb{C}$ and $\xi=(\xi_1,\xi_2)\in U\times U$ let
$$
PF(\xi)=\int_{U\times U}F(\eta) \,\mathrm{d}\Sigma_\xi(\eta)
$$
where
$$
\,\mathrm{d}\Sigma_\xi(\eta)=\delta(\xi_1+\xi_2-\eta_1-\eta_2)\delta(\phi(\xi_1)+\phi(\xi_2)-\phi(\eta_1)-\phi(\eta_2)) \mathrm{d}\eta.
$$
Notice that if $\mathbf{1}$ denotes the constant function equal to $1$ on $U\times U$ then
$$
P\mathbf{1}(\xi)=\mu*\mu(\xi_1+\xi_2,\phi(\xi_1)+\phi(\xi_2)),
$$
where the $S$-carried measure $\mathrm{d}\mu$ is given by
$$
\int\psi \,\mathrm{d}\mu := \int_U\psi(u,\phi(u))\,\mathrm{d}u.
$$ 
The nonvanishing curvature of the surface $S$ guarantees that $\mu*\mu$ is a bounded function on $\mathbb{R}^3$, and thus there exists a constant $c<\infty$ such that $P\mathbf{1}(\xi)\leq c$ for all $\xi\in U\times U$.

Next we choose the constant c in the definition of $Q$ to be $\|P\mathbf{1}\|_\infty$. Multiplying out the powers in the expression for $Q$ and using Fubini's theorem we obtain
\begin{eqnarray*}
\begin{aligned}
Q(t)&=c\int_{U}|g_t(\xi_1)|^2 \,\mathrm{d}\xi_1\int_{U}|g_t(\xi_2)|^2 \,\mathrm{d}\xi_2 \\
&\qquad \qquad -\int_{U^4}g_t(\xi_1)g_t(\xi_2)\overline{g_t(\eta_1)}\overline{g_t(\eta_2)}\delta(\xi_1+\xi_2-\eta_1-\eta_2)\delta(\phi(\xi_1)+\phi(\xi_2)-\phi(\eta_1)-\phi(\eta_2)) \,\mathrm{d}\eta \,\mathrm{d}\xi\\
&=c\int_{U^2}|G_t(\xi)|^2 \,\mathrm{d}\xi - \int_{U^4}G_t(\xi)\overline{G_t(\eta)} \,\mathrm{d}\Sigma_\xi(\eta)\mathrm{d}\xi\\
&=c\int_{U^2}|G_t(\xi)|^2 \,\mathrm{d}\xi-\int_{U^2}G_t(\xi)\overline{PG_t(\xi)}\,\mathrm{d}\xi,\\
\end{aligned}
\end{eqnarray*}
where $G_t=g_t\otimes g_t$.
Differentiating through the integral, using the self-adjointness of $P$ along with the fact that $$\frac{\,\mathrm{d}}{\,\mathrm{d}t}G_t(\xi)=-(\phi(\xi_1)+\phi(\xi_2))G_t(\xi),$$ we obtain
\begin{eqnarray*}
\begin{aligned}
Q'(t)
&=2\int_{U^2}(\phi(\xi_1)+\phi(\xi_2))G_t(\xi)\overline{PG_t(\xi)} \,\mathrm{d}\xi-2c\int_{U^2}(\phi(\xi_1)+\phi(\xi_2))|G_t(\xi)|^2 \,\mathrm{d}\xi,
\end{aligned}
\end{eqnarray*}
and so
\begin{eqnarray*}
\begin{aligned}
Q'(t)
&\leq 2\int_{U^2}(\phi(\xi_1)+\phi(\xi_2))G_t(\xi)\overline{PG_t(\xi)} \,\mathrm{d}\xi-2\int_{U^2}(\phi(\xi_1)+\phi(\xi_2))|G_t(\xi)|^2P\mathbf{1}(\xi) \,\mathrm{d}\xi\\
&=2\int_{U^4}(\phi(\xi_1)+\phi(\xi_2))G_t(\xi)\overline{G_t(\eta)} \,\mathrm{d}\Sigma_\xi(\eta) \mathrm{d}\xi-2\int_{U^4}(\phi(\xi_1)+\phi(\xi_2))|G_t(\xi)|^2 \,\mathrm{d}\Sigma_\xi(\eta) \mathrm{d}\xi.\\
\end{aligned}
\end{eqnarray*}
Now, interchanging the order of integration, using the fact that $\phi(\xi_1)+\phi(\xi_2)=\phi(\eta_1)+\phi(\eta_2)$ on the support of $\mathrm{d}\Sigma_\xi$, and interchanging variables of integration reveals that
$$
\int_{U^4}(\phi(\xi_1)+\phi(\xi_2))|G_t(\xi)|^2 \,\mathrm{d}\Sigma_\xi(\eta) \mathrm{d}\xi=
\int_{U^4}(\phi(\xi_1)+\phi(\xi_2))|G_t(\eta)|^2 \,\mathrm{d}\Sigma_\xi(\eta) \mathrm{d}\xi,
$$
and so
\begin{eqnarray*}
\begin{aligned}
Q'(t)
&\leq 2\int_{U^4}(\phi(\xi_1)+\phi(\xi_2))G_t(\xi)\overline{G_t(\eta)} \,\mathrm{d}\Sigma_\xi(\eta)\mathrm{d}\xi - \int_{U^4}(\phi(\xi_1)+\phi(\xi_2))|G_t(\xi)|^2 \,\mathrm{d}\Sigma_\xi(\eta) \mathrm{d}\xi\\
& \qquad \qquad -\int_{U^4}(\phi(\xi_1)+\phi(\xi_2))|G_t(\eta)|^2 \,\mathrm{d}\Sigma_\xi(\eta) \mathrm{d}\xi\\
&= - \int_{U^4}(\phi(\xi_1)+\phi(\xi_2))|G_t(\xi)-G_t(\eta)|^2 \,\mathrm{d}\Sigma_\xi(\eta) \mathrm{d}\xi \leq 0
\end{aligned}
\end{eqnarray*}
for all $t>0$. Arguing as above, but taking $k$ derivatives, reveals that
$$
(-1)^{k}Q^{(k)}(t)\geq \int_{U^4}(\phi(\xi_1)+\phi(\xi_2))^k|G_t(\xi)-G_t(\eta)|^2\,\mathrm{d}\Sigma_\xi(\eta)\mathrm{d}\xi\geq 0
$$
for every $k$.
\end{proof}
\subsection*{Remark}
Of course the flow $t\mapsto g_t$, which amounts to a simple damping of the function $g$ by an appropriate factor, is not diffusive. However, if the surface $S$ is replaced by a suitable compact manifold without boundary then it seems natural to return to diffusion. If $\mathrm{d}\sigma$ denotes surface measure on $\mathbb{S}^2$ then it was recently shown by Foschi \cite{Fo2013} that
$$
\int_{\mathbb{R}^3}|\widehat{g\mathrm{d}\sigma}|^4\leq\int_{\mathbb{R}^3}|\widehat{\mathrm{d}\sigma}|^4
$$ 
for all $g\in L^2(\mathbb{S}^2)$ satisfying $\|g\|_2=1$; i.e. that the Stein--Tomas restriction theorem for the sphere in $\mathbb{R}^3$ is extremised by constant functions. This raises the possibility that there is an underlying monotonicity phenomenon as the function $g$ (or $|g|^2$) diffuses under, for instance, the heat equation $\partial_t u=\Delta_{\mathbb{S}^2}u$. Here $\Delta_{\mathbb{S}^2}$ denotes the Laplace--Beltrami operator on $\mathbb{S}^2$. It seems plausible that this may follow from the representation formula
$
\|\widehat{g\mathrm{d}\sigma}\|_4^4=(2\pi)^3\langle G,PG\rangle,
$
where $G=g\otimes g$ and $P:L^2(\mathbb{S}^2\times\mathbb{S}^2)\rightarrow L^2(\mathbb{S}^2\times\mathbb{S}^2)$ is given by
$$
PG(x_1,x_2)=\frac{1}{|x_1+x_2|}\int_0^{2\pi}G(\rho_\theta(x_1,x_2)) \,\mathrm{d}\theta,
$$
where $\rho_\theta$ applied to $(x_1,x_2)$ rotates both $x_1, x_2$ on $\mathbb{S}^2$ clockwise through an angle $\theta$ about their midpoint. The corresponding representation formula in the context of the paraboloid (see \cite{FGH}) appears to be more elementary since it does not have the weight factor $|x_1+x_2|$ and the associated rotations $\rho_\theta$ are easily seen to be isometries.

\section{Remarks on the Strichartz estimates for the kinetic transport equation}\label{sect4}
It is well-known that the solution of the kinetic transport equation
\begin{equation*}
\partial_s f(s,x,v) + v \cdot \nabla_x f(s,x,v) = 0, \qquad f(0,x,v) = f^0(x,v)
\end{equation*}
for $(s,x,v) \in \mathbb{R} \times \mathbb{R}^d \times \mathbb{R}^d$,
satisfies estimates that are very similar to the Strichartz estimates \eqref{strichschrod} for the Schr\"odinger equation. These may be written as
\begin{equation} \label{e:Strichartz}
\|\rho(f^0)\|_{L^q_sL^p_x} \lesssim \|f^0\|_{L^a_{x,v}},
\end{equation}
where the macroscopic density $\rho(f^0)$ is given by
$$
\rho(f^0)(s,x)=\int_{\mathbb{R}^d}f(s,x,v) \,\mathrm{d}v = \int_{\mathbb{R}^d}f^0(x-vs,v) \,\mathrm{d}v.
$$
Necessary and sufficient conditions on the exponents $(a,p,q)$ for \eqref{e:Strichartz} to hold are
\begin{equation} \label{e:r1scaling}
q > a, \qquad p \geq a, \qquad \frac{2}{q} = d\bigg(1-\frac{1}{p}\bigg), \qquad  \frac{1}{a} = \frac{1}{2}\bigg(1 + \frac{1}{p}\bigg);
\end{equation}
see \cite{CP}, \cite{KT} and \cite{BBGL}. While quite similar to the Strichartz estimates for the Schr\"odinger equation, these inequalities contain some interesting phenomenological differences. The most apparent is the absence of an endpoint in \eqref{e:r1scaling}; see \cite{KT} and \cite{BBGL}. In this section we highlight a more subtle difference concerning the nature of the flows under which we might expect monotonicity properties.
Given the monotonicity phenomena in the context of the Strichartz estimates for the Schr\"odinger equation, a natural candidate for consideration is the case $p=q=\frac{d+2}{d}$ and $a=\frac{d+2}{d+1}$, namely
\begin{equation}\label{purenorm}
\|\rho(f^0)\|_{L^{\frac{d+2}{d}}_{s,x}} \leq C \|f^0\|_{L^{\frac{d+2}{d+1}}_{x,v}}.
\end{equation}
Notice that for $d=1, 2$ the exponent on the left hand side is an integer, making \eqref{purenorm} an analogue of the pure-norm Strichartz estimate for the Schr\"odinger equation.
By duality \eqref{purenorm} is equivalent to
\begin{equation}\label{purenormdual}
\|\rho^*(g)\|_{L^{d+2}_{x,v}} \leq C \|g\|_{L^{\frac{d+2}{2}}_{s,x}}
\end{equation}
where
$$
\rho^*(g)(x,v)=\int_{\mathbb{R}}g(s,x+vs) \,\mathrm{d}s,
$$
which may be interpreted as a space-time X-ray transform estimate. We remark that in the particular case $d=1$, the inequality \eqref{purenorm} is effectively self-dual, as may be seen by suitably interchanging the roles of the scalar parameters $s,x,v$. A sharp form of \eqref{purenormdual} was found recently by Drouot \cite{Drouot}, who identified the functions $$g(s,x)=\frac{1}{1+s^2+|x|^2}$$ as extremisers. Shortly afterwards Flock \cite{Flock}, following work of Christ \cite{ChristRadon}, established that all extremisers take this form up to the symmetries of the inequality. Of course this characterisation constrains the flows that we might consider in this context, eliminating, for example, a heat-flow approach in the spatial variable as in \cite{BCT}. However, there is some evidence to suggest that certain nonlinear space-time diffusions may be appropriate. To see this we consider a more general $k$-plane transform
$$
\rho_k^*(g)(x,v)=\int_{\mathbb{R}^k}g(s,x+vs) \,\mathrm{d}s,
$$
where now $x\in\mathbb{R}^{d+1-k}$ and $v$ is a $(d+1-k)\times k$ matrix which we identify with $\mathbb{R}^{k(d+1-k)}$ equipped with Lebesgue measure. Of course $\rho^*_1(g)=\rho^*(g)$. It was shown in \cite{Flock}, following \cite{ChristRadon} in the case $k=d$, that $$\|\rho_k^*(g)\|_{L^{d+2}_{x,v}}=C\|T_{k,d+1}(g)\|_{L^{d+2}(\mathcal{M}_{k,d+1})},$$ where $T_{k,d+1}$ denotes the classical $k$-plane transform on $\mathbb{R}\times\mathbb{R}^d\cong\mathbb{R}^{d+1}$ and $\mathcal{M}_{k,d+1}$ the Grassmann manifold of all affine $k$-planes in $\mathbb{R}^{d+1}$. The extremisers of the associated inequality
\begin{equation}\label{purenormdualk}\|\rho_k^*(g)\|_{L^{d+2}_{x,v}}=C\|T_{k,d+1}(g)\|_{L^{d+2}(\mathcal{M}_{k,d+1})}\lesssim \|g\|_{L^{\frac{d+2}{k+1}}(\mathbb{R}^{d+1})}\end{equation}
take the form
\begin{equation}\label{extremisersDrouot} g(s,x)=\frac{1}{(1+s^2+|x|^2)^{\frac{k+1}{2}}},
\end{equation}
again up to symmetries; see \cite{ChristRadon}, \cite{Drouot} and \cite{Flock}.
Interpolating \eqref{purenormdualk} with the trivial $L^1$ estimate for $T_{k,d+1}$ then gives
\begin{equation}\label{purenormdualkp}\|T_{k, d+1}(g)\|_{L^{q}(\mathcal{M}_{k,d+1})}\lesssim \|g\|_{L^{p}(\mathbb{R}^{d+1})}\end{equation}
where $\frac{d+1}{p}=k+\frac{d+1-k}{q}$, $1\leq p\leq \frac{d+2}{k+1}$. By combining elements of \cite{BLoss} and \cite{CCL} it follows, at least when $q=k=2$, that the corresponding functional
\begin{equation}\label{functional}
\mathcal{F}(g)=c\|g\|_{L^{p}(\mathbb{R}^{d+1})}^q-\|T_{k, d+1}(g)\|_{L^{q}(\mathcal{M}_{k,d+1})}^q,
\end{equation}
with $c=c_{*}(k,q)$ denoting the optimal constant in \eqref{purenormdualkp}, is monotone with respect to a certain fast diffusion.
More specifically, applying Drury's identity (as in \cite{BLoss}) we may write
\begin{equation}\label{druryidentity}
\|T_{2,d+1}(g)\|_{L^2(\mathcal{M}_{2,d+1})}^2 = C\int_{\mathbb{R}^{d+1}}\int_{\mathbb{R}^{d+1}}\frac{g(x)g(y)}{|x-y|^{d-1}} \,\mathrm{d}x\mathrm{d}y,
\end{equation}
for some constant $C$, which upon applying the Hardy--Littlewood--Sobolev monotonicity theorem of Carlen, Carrillo and Loss \cite{CCL} (see also the variant in \cite{Dolbeault}) yields the following:
\begin{theorem}\label{ccl} Suppose $d>1$, $q=k=2$, $m=\frac{d+1}{d+3}$ and that $g\in L^{\frac{2(d+1)}{d+3}}(\mathbb{R}^{d+1})$ is nonnegative and of compact support.
If $u:[0,\infty)\times \mathbb{R}^{d+1}\rightarrow [0,\infty)$ satisfies
\begin{equation}\label{fast}
\partial_t u=\Delta(u^{m});\;\;\;u(0,\cdot)=g,
\end{equation} then
$t\mapsto \mathcal{F}(u(t,\cdot))$ is nonincreasing.
\end{theorem}
The relevance of this particular fast diffusion is indicated by the nature of the extremisers \eqref{extremisersDrouot} and the asymptotic profiles, or so-called Barenblatt profiles, for solutions of \eqref{fast}; see \cite{CCL} for further discussion. 
It is conceivable that alternative forms of Drury's identity may allow a variant of Theorem \ref{ccl} in the case $k=1$, $q=d+2$, yielding monotonicity phenomena for the Strichartz inequalities \eqref{purenormdual}, \eqref{purenorm}. We do not pursue this further here.

\end{document}